\documentclass[11pt]{amsart}
\usepackage{amsmath,amssymb,latexsym}
\usepackage[T1]{fontenc}
\usepackage{enumerate}

\def\N{\mathbb N}

\def\stab{\mathrm{Stab}}
\def\st{\mathrm{st}}
\def\U{\mathfrak{U}}

\def\H{\mathcal{H}}

\theoremstyle{plain}
\newtheorem{theorem}{Theorem}[section]

\newtheorem{proposition}[theorem]{Proposition}
\newtheorem{fact}[theorem]{Fact}
\newtheorem{lemma}[theorem]{Lemma}
\newtheorem{cor}[theorem]{Corollary}

\theoremstyle{definition}
\newtheorem{definition}[theorem]{Definition}

\newtheorem{example}{Example}

\def\Ind#1#2{#1\setbox0=\hbox{$#1x$}\kern\wd0\hbox to 0pt{\hss$#1\mid$\hss}
\lower.9\ht0\hbox to 0pt{\hss$#1\smile$\hss}\kern\wd0}

\def\Notind#1#2{#1\setbox0=\hbox{$#1x$}\kern\wd0\hbox to 0pt{\mathchardef
\nn="3236\hss$#1\nn$\kern1.4\wd0\hss}\hbox to 0pt{\hss$#1\mid$\hss}\lower.9\ht0
\hbox to 0pt{\hss$#1\smile$\hss}\kern\wd0}

\title[On compactifications and product-free sets]{On compactifications and product-free sets}
\author{Daniel Palac\'in}
\thanks{I would like to thank Emmanuel Breuillard and Gabriel Conant for helpful conversations and comments. 
\newline Research partially supported by the program MTM2017-86777-P and the European Research Council grant 338821.}
\address{Abteilung f\"ur Mathematische Logik, Mathematisches Institut, Albert-Ludwig-Universit\"at Freiburg; Ernst-Zermelo-Str. 1, 79104 Freiburg, Germany}
\email{palacin@math.uni-freiburg.de}

\begin{document}

\begin{abstract}
A subset of a group is said to be product-free if it does not contain three elements satisfying the equation $xy=z$. We give a negative answer to a question of Babai and S\'os on the existence of large product-free sets in finite groups by model theoretic means. This question was originally answered by Gowers. Furthermore, we give a natural and sufficient model theoretic condition for a group to have a large product-free subset, as well as a model theoretic account of a result of Nikolov and Pyber on triple products. 
\end{abstract}

\maketitle

\section{Introduction}

Babai and S\'os \cite{BaSo} asked whether there exists a constant $c>0$ such that every finite group of order $n$ has a product-free set of size at least $cn$, where a product-free set of a group is a subset that does not contain three elements $x,y$ and $z$ satisfying $xy=z$. While they proved in \cite[Cororally 7.8]{BaSo} that a solvable group of order $n$ has a product-free set of size at least $2n/7$, a negative answer to their question was obtained by Gowers \cite{Gow}. 

The key parameter of a finite group that plays a fundamental role in Gowers proof is the minimum dimension of a non-trivial unitary representation of the given finite group. Using a result of Frobenius asserting that $\mathrm{PSL}_2(q)$ has no non-trivial unitary representation of dimension less than $(q-1)/2$, as well as facts concerning quasirandom bipartite graphs, Gowers showed that for a sufficiently large $q$ the group $\mathrm{PSL}_2(q)$ has no product-free subset of size $cn^{8/9}$, where $n$ is the order of the group $\mathrm{PSL}_2(q)$ and $c>0$ is a constant given beforehand. 

From a more abstract point of view, Gowers related the existence of low dimensional unitary representations with the existence of product-free sets of large density, by providing an upper bound for the size of a product-free set in terms of the minimum dimension of a non-trivial unitary representation. This is the content of \cite[Theorem 3.3]{Gow}.

An easy but remarkable consequence of this theorem, noticed by Nikolov and Pyber in \cite[Corollary 1]{NikPyb}, is the following: Given a group $G$ of order $n$ without a non-trivial unitary representation of dimension less than $d$, we have that $G=A^3$ for any subset $A$ of $G$ of size at least $n/\sqrt[3]{d}$.  We refer to \cite{Bre} for a distinct and shorter approach based on the non-abelian Fourier transform due to Breuillard. 

From the point of view of model theory, these results resemble some kind of amalgamation principle for definable ternary relations. This intuition comes from the work of Cherlin and Hrushovski \cite{ChHr} which later played a major role in the development of model-theoretic simplicity, see for instance \cite[Lemma 6.1.14]{ChHr} or its generalization to simple theories due to Pillay, Scanlon and Wagner \cite{PiScWa}. These latter results are obtained due to the existence of the so called stabilizer subgroup, a suitable normal subgroup of small index. In the context of nonstandard finite groups, which can be seen as suitable model theoretic limits of finite groups, the existence of such a subgroup was proven by Hrushovski \cite{Hrus1}, who  used it to establish links between model theory and finite combinatorics. Roughly speaking, Hrushovski's stabilizer theorem corresponds to a nonstandard version of a refinement of a theorem of Sanders \cite{San} from finite combinatorics; we refer to \cite{MaWa} for a model theoretic treatment of Sanders' theorem.

The aim of this paper it to give a model theoretic account of a qualitative analogue of the aforementioned results of Gowers, as well as the corollary obtained by Nikolov and Pyber, answering a question of Hrushovski, see \cite[Remark 3.3]{Hrus3}. 

We first see in Section \ref{S:PF} that the existence of the stabilizer subgroup allows us to answer negatively the original question of Babai and S\'os, giving an alternative proof to the one of Gowers. The reason for this is that in the nonstandard setting, the existence of a produt-free set of a given density is equivalent to the existence of a subgroup of bounded index. This result has the following translation to the finite setting:

\begin{theorem}
For any $c>0$ and any $m\ge 1$, there exists some $k=k(c,m)$ such that the following holds. Suppose that $G$ is a finite group of order $n$  with a product-free set $A$ of size at least $cn$. Then there are symmetric subsets 
$$
X_{m}\subseteq X_{m-1} \subseteq \ldots \subseteq X_1 \subseteq (AA^{-1})^2
$$
containing the identity and properly contained in $G$ with the property that $X_{i+1}^2 \subseteq X_i$ and such that $G$ is covered by at most $k$ many translates of $X_m$.
\end{theorem}

Note that the converse also holds. Namely, the existence of the set $X_3$ yields the existence of a product-free set of size $n/k$, since any set $X$ has a translate which is product-free whenever $XX^{-1}X$ is properly contained in $G$.  

This statement can be easily deduced using the stabilizer theorem due to Hrushovski. Furthemore, using the work of Gowers \cite{Gow}, or the structure theorem of approximate subgroups due to Breuillard, Green and Tao \cite{BGT}, we see that the above statement is also equivalent to the existence of abelian-by-bounded quotients. In contrast, regarding finite groups of a given exponent, a simpler argument yields the existence of bounded index subgroups. In the sequel this appears as Corollary \ref{C:FinExp}.

\begin{theorem}
For any $c,\epsilon>0$ and any $r$, there exists some $k=k(c,\epsilon,r)$ such that the following holds. Suppose that $G$ is a finite group of order $n$ and exponent $r$ and let $A$ be a product-free set of size at least $cn$. Then there is a proper subgroup $H$ of $G$ of index at most $k$ which is contained in $(AA^{-1})^2$.
%which satisfies the following properties:
%\begin{enumerate}[$(a)$]
%\item it is contained in $(AA^{-1})^2$ with $|H\setminus AA^{-1}|<\epsilon |G|$, and
%\item some coset $C$ of $H$ is contained in $AA^{-1}A$ with $|C\cap A|\ge O_{c,\epsilon,m}(1)|G|$.
%\end{enumerate}
\end{theorem}
 
The explanation from a model theoretic point of view has to do with the existence of a non-trivial definable group compactification associated to a given ultraproduct of finite groups. In fact, the existence of such a non-trivial group compactification also explains why in the finite exponent case one always obtains a proper subgroup of finite index. The reason is that in this case, the group compactification will be a profinite group.  

The notion of definable group compactification is carefully explained and studied in Section \ref{S:DefCom} and  \ref{S:FinExp}, which contain the main model theoretic content of the paper. In particular, in Section \ref{S:DefCom} we prove our main model theoretic results by showing in Theorem \ref{T:Main} that such a definable group compactification is non-trivial if and only if there exists a large product-free subset. Furthermore, in Theorem \ref{T:Main} we also prove the nonstandard version of Gowers Theorem \cite[Theorem 3.3]{Gow} as well as the derivation of exponent $3$ obtained by Nikolov and Pyber.

\section{A non-standard formulation}

As discussed in the introduction, for our purposes it is convenient to consider nonstandard finite groups. We assume the reader is familiarized with nonstandard methods, but we recall here some of the basic properties making special emphasis on definable issues.

In the literature one can find different ways to consider nonstandard mathematical objects. One natural way is to work in a nonstandard model of set theory, which is precisely the approach considered in \cite{Pil1}. However, this elegant construction to handle nonstandard objects might not be the easiest one for those readers not versed in mathematical logic. Consequently, we give a more specific procedure to consider nonstandard structures, assuming that the reader is familiar, up to some extent, with model theory. We first recall some basic notions. For a detailed treatment, we refer to Section 4 of \cite{vdD}, as well as to the Appendix of \cite{vdDGold} for another exposition. 

\subsection{Logical preliminaries} A {\em structure} $M$ consists of a family $(S_i)_{i\in I}$ of nonempty sets $S_i$, and of a family $(R_j)_{j\in J}$ of relations $R_j$ of $S_{i_1}\times \ldots \times S_{i_m}$ on these sets, with the finitely many indices $i_1,\ldots,i_m$ depending on $j$. We think of the sets $S_i$ as the underlying {\em sorts} or {\em basic sets} of the structure, and the relations $R_j$ as its {\em primitives}. It is assumed that equality is a
primitive binary relation in every sort. From a logical point of view, we regard the structure $M$ as an $I$-sorted structure in the language $L$ consisting of symbols for the primitive relations $R_j$ with $j\in J$.

For instance, a group or a ring can be seen as a $1$-sorted structure with the domain as a basic set and the graph of their algebraic operations as primitives. %A fundamental, and slightly more elaborated, example for our purposes is the following. 

The {\em definable sets} of a structure $M$ are those subsets of Cartesian powers of $M$ that are obtained from the primitive relations in finitely many steps using the following operations: intersection, union, complement, Cartesian product and image under a co-ordinate projection. 
Alternatively, the definable sets are the solution sets of formulas $\phi(x_1,\ldots,x_n)$ in $L$, where each variable $x_i$ is of a specific sort. Furthermore, given an $I$-sorted subset $A=(A_i)_{i\in I}$ of $M$, we can expand the language by naming parameters, that is to say for each element $a_i\in A_i$ we add as a primitive the relation $\{a_i\}$. In this way we expand the language $L$ to $L(A)$ and so the $L$-structure $M$ is tautologically expanded to an $L(A)$-structure, which we also denote by $M$. In particular, we can now speak of $A$-definable sets of $M$, that is to say definable sets of $M$ defined with parameters coming from $A$. From now on, by a definable set in the structure $M$ we mean a definable set with parameters from $M$.

As usual in nonstandard methods, to ensure that our ambient structure enjoys of some model-theoretic compactness, we may take it to be {\em $\kappa$-saturated} for a suitable infinite cardinal $\kappa$.  We remind here that $\kappa$-saturation means that any intersection of less than $\kappa$ many definable subsets is non-empty, provided that all finite sub-intersection are. Let us remark that any infinite structure is elementarily equivalent to a $\kappa$-saturated structure for all $\kappa \ge |L|$, by compactness.
For many arguments, it is enough to take $\kappa=\aleph_1$ in which case one can take $M$ to be an ultraproduct of finite structures.

Assuming that the structure $M$ is $\kappa$-saturated, we can thus consider {\em type-definable sets} which are by definition the intersection of less than $\kappa$ many definable sets. In particular, a definable set is type-definable and in fact, an easy compactness argument yields that a set is definable if and only if its complement and itself are both type-definable. If $X$ is a type-definable set, we usually denote by $X(x)$ the {\em partial type} defining it, that is to say the collection $\{X_i(x)\}_{i\in I}$ of formulas $X_i(x)$ such that $X = \bigcap_i X_i$,  and {\it vice versa}. %In particular, given a complete type $p(x)$ over a small set of parameters we denote by $p$ its set of realizations. 
We remind that if the structure $M$ is very saturated, then a set is type-definable over a small set of parameters $A$ of $M$ if and only if it is type-definable and invariant under the group of automorphism $\mathrm{Aut}(M/A)$ of $M$ fixing the set $A$ pointwise. By a small set of parameters we mean a subset whose size is smaller than the degree of saturation of $M$.

To finish this introductory subsection, we recall the concept of relative connected component as well as the logic topology, since both play an essential role in the paper. For a detailed explanation we refer the reader to Section 2 of \cite{Pil0}. Fix a $\kappa$-saturated $L$-structure $M$, for a sufficiently large $\kappa$, and let $G$ be a definable group in $M$. A type-definable subgroup $H$ is said to have {\em bounded index} if the coset space $G/H$ has size less than $\kappa$. Equivalently, if $G$ and $H$ are defined over $M_0$, then $|G/H|<2^{|L|+|M_0|}$. For each small set $A$ of parameters there exists a smallest type-definable over $A$ subgroup of $G$ of bounded index. This group is in addition normal and we denote it by $G_A^{00}$. One can equip the coset space $G/H$ with a compact Hausdorff topology called the {\em logic topology}. A subset of $G/H$ is closed in the logic topology if and only if its pre-image under the natural projection is a type-definable set. %Note that the compactness of $G/H$ comes from model theoretic compactness. 
Furthermore, when $H$ is a normal subgroup of $G$, then $G/H$ is a compact Hausdorff topological group. 

\subsection{Ultraproducts} Let $\U$ be a non-principal ultrafilter on $\N$, that is to say a collection of infinite subsets of $\N$ closed under intersections and with the property that either a set or its complement belongs to $\U$. We say that a property $Q(n)$ of $\N$ holds for $\U$-almost all $n$ if the set of natural numbers satisfying it belongs to $\U$. 
 
Let $(M_n)_{n\in\N}$ be an infinite sequence of $L$-structures in a countable language $L$. We write $\prod_\U M_n$ to denote the ultraproduct of $(M_n)_{n\in\N}$ with respect to $\U$, which is by definition the direct product $\prod_{n\in\N} M_n$ modulo the equivalence relation $\sim_\U$ given by 
$$
(a_i)_{i\in\N} \sim_\U (b_i)_{i\in\N} \ \Leftrightarrow \ a_i=b_i \text{ for $\U$-almost all $n$}.
$$ 
The ultraproduct of $L$-structures is again an $L$-structure which in addition is $\aleph_1$-saturated. Another fundamental feature of ultraproducts is the following transfer principle.

\begin{theorem}[\L os' Theorem]
Let $\phi(x_1,\ldots,x_m)$ be an $L$-formula and let $a_1,\ldots,a_m$ be elements of the ultraproduct $\prod_\U M_n$ of $L$-structures $(M_n)_{n\in \N}$ with representatives $(a_{1,n})_{n\in \N},\ldots, (a_{m,n})_{n\in\N}$. Then $\phi(a_1,\ldots,a_m)$ holds if and only if $\phi(a_{1,n},\ldots,a_{m,n})$ holds for $\U$-almost all $n$.
\end{theorem}

In this paper we are mainly interested in ultraproducts of finite groups, in which case $M_n$ is a $1$-sorted structure consisting of a sort for the domain of a finite group and the graph of the group operation as a primitive. In this case, the language is the pure language of groups $L_{\mathrm{gr}}=\{\cdot, ^{-1}\}$. As we shall see later, it is convenient to enrich the language and consider finite groups with some additional structure.

\begin{definition}
A subset $X$ of an ultraproduct $\prod_\U G_n$ of finite groups $(G_n)_{n\in \N}$ is said to be {\em internal} if there is a sequence $(X_n)_{n\in\N}$ such that each $X_n$ is a subset of $G_n$ and $X = \prod_\U X_n$. 
\end{definition}
We see that every definable set in the group language $L_{\mathrm{gr}}$ is internal, but the converse is not true. The collection of all internal subsets forms a Boolean algebra which is left and right translation-invariant under the group multiplication. 

Let $G$ be the ultraproduct $\prod_\U G_n$ of finite groups $(G_n)_{n\in\N}$. We define a finitely probability measure on the Boolean algebra 
of its internal sets, which is left and right translation-invariant under the group multiplication. To do so, consider the ultraproduct $\mathbb R^* = \prod_\U \mathbb R$ of the ordered field of real numbers, which is an ordered real closed field, and note that $\mathbb R$ embeds into $\mathbb R^*$. We call the elements of $\mathbb R^*$ nonstandard real numbers. We can assign to each internal subset $A=\prod_\U A_n$ the nonstandard real number $|A|$ corresponding to the equivalence class of $(|A_n|)_{n\in \N}$ in $\mathbb R^*$. A nonstandard real number $r\in \mathbb R^*$ with the property that it is bounded above or below by some real number is called finite. The standard part map $\mathrm{st}$ assigns to each finite nonstandard real number $a^*$ the unique real number $a$ such that $|a^*-a|<\frac{1}{k}$ for every natural number $k$. Finally, define the measure of $A$ by 
$$
\mu(A) = \mathrm{st}\left(\frac{|A|}{|G|}\right).
$$
This measure is easily seen to be stable under left and right multiplication, {\it i.e.} $\mu(gA)=\mu(A)=\mu(Ag)$ for any $g\in G$ and any internal subset $A$. 

An internal subset of $G$ is said to be a {\em null set} if it has measure zero, and non-null otherwise. We also say that a subset, not necessarily internal, is {\em wide} if all its internal supersets are non-null. The collection of internal null sets for an ideal, that is to say it is a proper collection of subsets which is closed under containment and finite unions. One of the main technical problems we encounter is that neither this ideal nor the measure are {\it a priori} invariant under automorphisms. To circumvent this issue one can expand the language, as we explain in the next subsection.

To finish this subsection, we give an easy but relevant fact on the counting measure.

\begin{lemma}\label{L:Intersection}
Let $G$ be an ultraproduct of finite groups and let $A$ and $B$ be two internal subsets of it. Then there exists some element $g$ of $AB^{-1}$ such that 
$$
\mu(A\cap gB)\mu(AB^{-1}) \ge \mu(A)\mu(B).
$$
\end{lemma}
\begin{proof}
Let $G=\prod_\U G_n$ be an ultraproduct of finite groups and suppose that $A=\prod_\U A_n$ and $B=\prod_\U B_n$. Observe first that for every $n$, there is always an element $u_n\in A_nB_n^{-1}$ such that $A_n\cap u_nB_n$ has maximal size among all subsets of this form. Furthermore, an easy counting argument yields
$$
|A_n||B_n| = \sum_{x\in A_nB_n^{-1}} \left|\{(a,b)\in A_n\times B_n : a=xb\} \right| = \sum_{x\in A_nB_n^{-1}} |A_n\cap xB_n|
$$
and so $|A_n||B_n| \le |A_nB_n^{-1}||A_n\cap u_nB_n|$. Setting $u$ to be the ultraproduct of $(u_n)_{n\in \N}$, we see that $|A||B| \le |AB^{-1}||A\cap uB|$. Dividing out by $|G|^2$ and taking the standard part we obtain the desired equation.
\end{proof}

\subsection{Nonstandard finite groups}  As pointed out before, it is convenient to consider richer languages in a way that $(i)$ internal sets become definable and $(ii)$ the measure described above is invariant. To do so, we ``canonically'' expand the structure of finite groups.

Let $G$ be a finite group. Consider the structure consisting of one sort for the finite group $G$ with the graph of its group operation as a primitive, another sort for the ordered field of the real numbers $\mathbb R$ with primitives for the graph of the order and the algebraic operations. Furthermore, we also add a sort for the power-set $\mathcal P(G^m)$ of each Cartesian power of the group $G$, with a primitive $\in_{m}$ on $G^m\times \mathcal P(G^m)$ defining the membership relation: a pair $(x,u)$ of $G^m \times \mathcal P(G^m)$ is $\in_m$-related if and only if $x\in u$. Finally, for each $m$ we add as a primitive the graph of the cardinality function from $\mathcal P(G^m)$ to $\mathbb R$. We denote by $L^*$ the language consisting of sorts and symbols for these primitives.

Seeing a finite group $G$ as a ``finite $L^*$-structure'', note that any subset of the Cartesian power $G^m$ of $G$ is definable by the same $L^*$-formula, varying the parameters. Namely, a subset $A$ of $G^m$ is defined by the formula $x\in_m A$, where we identify the set $A$ with its corresponding name in the sort $\mathcal P(G^m)$. 

We remark that any ultraproduct of finite groups can also be expanded to an $L^*$-structure  in a way that the resulting $L^*$-structure is an ultraproduct of finite $L^*$-structures. Furthermore, the paragraph  above yields that a subset of any ultraproduct $G$ of finite groups is internal if and only if it is $L^*$-definable. In particular, this yields that the logic compactness theorem as well as \L os' Theorem hold for internal sets, see \cite[Appendix A]{BGT} for another approach.

Along the paper, it may be convenient to consider saturated elementary extensions of an ultraproduct of finite $L^*$-structures. This motivates the following terminology\footnote{This differs from the one of Pillay \cite[Definition 2.1]{Pil1}. In fact, any saturated nonstandard finite group in our terminology will have the structure of a nonstandard finite group in the sense of Pillay, and {\it vice versa}.}.

\begin{definition}
By {\em nonstandard finite group} we mean an infinite group $G$ which as an $L^*$-structure is elementarily equivalent to an ultraproduct of finite $L^*$-structures. 
\end{definition}

Before finishing the subsection, let us remark that working in the $L^*$-language there is no harm in passing from one model to another, since elementary properties are precisely inherited from the finite structures. More precisely, notice that a definable set $A$ in a nonstandard finite group corresponds to the solution set of the $L^*$-formula $x\in_m A$, for a suitable $m$. Moreover, the graph of the cardinality function is $L^*$-definable and hence the value of the counting measure of a set only depends on the formula defining it, not on the model we are working on. This allows us to work without loss of generality in a sufficiently saturated nonstandard finite group. In this situation, the ideal of definable null sets, which is an $\mathrm{S}1$-ideal in the sense of \cite[Definition 2.8]{Hrus1}, is $\emptyset$-invariant under automorphisms of the ambient structure and consequently, the work of Hrushovski around stabilizers applies. 

\section{Stabilizers and product-free sets} \label{S:PF}

Recall that we denote the set of realizations of a collection of formulas $X(x)$ by $X$, and {\it vice versa}. In particular, we write $p$ for the set of realizations of a complete type $p(x)$. Hence, we say that a partial type is wide if its set of realizations is. Note that a compactness argument yields that any wide partial type can be extended to a complete wide type over any set of parameters, since the collection of definable sets of measure zero form an ideal. As usual, we denote the space of types that concentrate on a definable set $X$ by $S_X(M)$. 

\subsection{Hrushovski's stabilizer} Let $G$ be a nonstandard finite group definable in a very saturated structure $\bar M$. By definable we mean in the sense of the expanded language $L^*$.

 Given a complete type $p(x)\in S_G(M)$ over a small model $M$ we define the set
$$
\st(p)=\{g\in G: gp\cap p \text{ is wide}\}.
$$ 
We see that this set contains the identity element of $G$ and it is symmetric, in the sense that it is closed under taking inverses. The stabilizer of $p$ is the group generated by the set $\st(p)$. We denote it by $\stab(p)$. 

Now, in view of Theorem 3.5 and Corollary 3.6 of \cite{Hrus1} we obtain the following result. In fact, in our framework it is more convenient to apply Theorem 2.12 (taking $X=G$) together with Proposition 2.13 and 2.14 of \cite{MOS}, which is a variant of the aforementioned results of Hrushovski. In any case, we may conclude:

\begin{theorem}\label{T:Udi}
Let $G$ be a nonstandard finite group and let $p(x)\in S_G(M)$ be a wide type over a small model $M$ of the ambient model. Then
$$
\stab(p)=(pp^{-1})^2 = \st(p)^2
$$
is a normal subgroup of bounded index and in addition the following holds:
\begin{enumerate}[$(i)$]
\item the set $pp^{-1}p$ is the coset $\stab(p)a$ of $\stab(p)$ with $a\in p$, and
\item the set $\stab(p)\setminus \st(p)$ is contained in a union of definable sets of measure zero with parameters over $M$.
\end{enumerate} 
In fact, the stabilizer $\stab(p)$ is $G_M^{00}$, the smallest type-definable subgroup of $G$ of bounded index defined over $M$.
\end{theorem}
%Here, by bounded we mean smaller than the degree of saturation of the ambient model. %Note that if there is a type-definable over $M$ subgroup of bounded index, then there exists a smallest one. We denote it by $G_M^{00}$.  
%{\it A priori} a nonstandard finite group $G$ might not have a smallest type-definable subgroup of bounded index, as witnessed in a non-principal ultraproduct of finite extra-special $p$-groups with $p>2$. Thus, the 

Using this we get a weaker version of \cite[Proposition 2.2]{PiScWa} in the nonstandard context, see also \cite[Theorem 4.7]{MaTe}. 

\begin{proposition}\label{P:Amalg}
Let $G$ be a nonstandard finite group and assume that it has a smallest type-definable bounded index subgroup $H$. If  $H$ is type-defined with parameters over a small model $M$ and $p_1(x),p_2(x),p_3(x)\in S_G(M)$ are wide types concentrated in $H$, then there are elements $a_i\in p_i$ such that $a_1a_2=a_3$. Furthermore, we then have $H=p_1p_2p_3$.
\end{proposition}
\begin{proof}
Note first that $H=G_{M'}^{00}$ for any small model $M'$ containing the parameters of $H$; in particular, this holds for $M$. Using Lemma \ref{L:Intersection}, by compactness we can find an element $g$ of $G$ such that $p_3\cap g p_2$ is wide. Let $N$ be an elementary extension of $M$ containing this element $g$ and let $q(x)\in S_G(N)$ be wide extensions of $(p_3\cap g p_2)(x)$. Note that $q$ is a subset of $H$, as so is $p_3$. On the other hand, the right translate $p_1g^{-1}$ of $p_1$ is wide, since the ideal of measure zero sets is preserved under right translation. Let $r(x) \in S_G(N)$ be some wide extension of the wide type $(p_1 g^{-1})(x)$ and note that the set $r$ is contained in $H$ as well, since $g\in p_3p_2^{-1}\subseteq H$. As $H=G_N^{00}$, applying Theorem \ref{T:Udi} we see that $qq^{-1}$ contains $r$ and so $p_3p_2^{-1}g^{-1} \cap p_1g^{-1} \neq \emptyset$. Consequently, we also get that $p_1p_2\cap p_3\neq \emptyset$.

For the second part, let $h\in H$ be an arbitrary element and consider the type-definable wide set $hp_3^{-1}$. Let now $M'$ denote a small model containing $M$ and $h$, and let $p_3'(x)\in S_G(M')$ be a wide type extending $(hp_3^{-1})(x)$. Take some arbitrary wide types $p_1'(x),p_2'(x)\in S_G(M')$ extending $p_1(x)$ and $p_2(x)$ respectively. Applying the first part of the statement, we find some elements $b_i\in p_i'$ such that $b_1b_2=b_3$. By construction, we get some $a\in p_3$ such that $b_3=ha^{-1}$ and so $h=b_1b_2a$, yielding that $H = p_1p_2p_3$. 
\end{proof}

The following fundamental observation relates product-free sets with the stabilizer subgroup, yielding that product-free types are precisely those types which are entirely contained in a proper coset of its stabilizer.
 
\begin{lemma}\label{L:StabPF}
Let $G$ be a nonstandard finite group and let $p(x)\in S_G(M)$ be a wide type over a small model of the ambient model. Then the following are equivalent:
\begin{enumerate}[$(a)$]
\item The set $p$ is product-free.
%\item The set $p$ is product-poor.
\item The set $pp^{-1}p$ is a proper coset of $\stab(p)$.
\item The set $(pp^{-1})^m p$ is product-free for each $m\ge 0$. 
\item The set $p$ is contained in a proper coset of $\stab(p)$.
%\item The stabilizer $\stab(p)$ of $p$ is a proper subgroup of $G$.
\end{enumerate} 
In particular, we have that $\stab(p)$ is a proper subgroup if and only if $p$ is product-free.
\end{lemma}
\begin{proof}
Assume that $p$ is product-free. By Theorem \ref{T:Udi} we know that $pp^{-1}p$ is a coset of $\stab(p)$. Therefore, to obtain $(b)$ it is enough to show that $pp^{-1}p\neq \stab(p)$. Suppose to get a contradiction that $pp^{-1}p=\stab(p)$. As the identity element belongs to $pp^{-1}$, we then have that $p$ is contained in $\stab(p)$. Furthermore, since the set $\stab(p)\setminus \st(p)$ is contained in a union of $M$-definable sets of measure zero by Theorem \ref{T:Udi}, we see that $p$ is a subset of $\st(p)$. Thus we get that $p$ is contained in $pp^{-1}$ and so $p^2\cap p \neq\emptyset$, a contradiction.

To see $(b)\Rightarrow(c)$, suppose that $pp^{-1}p$ is a proper right coset, say $\stab(p)u$, of $\stab(p)$. It then follows that $p$ is contained in $\stab(p)u$ and so
$$
(pp^{-1})^m p \subseteq \left(\stab(p)u \left( \stab(p)u \right)^{-1} \right)^m \stab(p)u = \stab(p)u.
$$ 
Thus $(pp^{-1})^m p$ is contained in the coset $\stab(p)u$ and so it is product-free. Finally, taking $m=0$ we get $(c)\Rightarrow (d)$ and obviously $(d)$ implies $(a)$, which finishes the proof.
\end{proof}

 We can now deduce (a strengthening) of Theorem 1.1.
\begin{cor}
For any $c,\epsilon>0$ and any $m\ge 1$, there are some $k=k(c,\epsilon,m)$ and $\eta=\eta(c,\epsilon,m)$ such that the following holds. Suppose that $G$ is a finite group of order $n$  with a product-free set $A$ of size at least $cn$. Then there are symmetric subsets 
$$
X_{m}\subseteq X_{m-1} \subseteq \ldots \subseteq X_1 \subseteq (AA^{-1})^2
$$
containing the identity with the property that $X_{i+1}^2 \subseteq X_i$ and such that the group $G$ is covered by at most $k$ many translates of $X_m$. Furthermore, we have that $|X_1\setminus AA^{-1}|<\epsilon n$ and that there is some $x\in A$ such that $X_1x$ is contained in $AA^{-1}A$ and $|X_mx\cap A|\ge \eta n$.
\end{cor}
\begin{proof}
Otherwise, negating quantifiers there are constants $c,\epsilon>0$ and a positive integer $m$ such that for each $n$ we can find a finite group $G_n$ with a product-free subset $A_n$ of size at least $c|G_n|$ witnessing a counterexample. Fix a non-principal ultrafilter $\U$ on $\N$ and let $G = \prod_\U G_n$ be  the ultraproduct of $(G_n)_{n\in\N}$ with respect to $\U$, seen as a nonstandard finite group in a structure $M$. Let $A=\prod_\U A_n$, an internal product-free subset of $G$, and note that $|A|\ge c|G|$. 

Consider some very saturated elementary extension $M^*$ of $M$. Let $G^*$ and $A^*$ be the interpretation in $M^*$ of the formulas $G(x)$ and $A(x)$, and notice that $A^*$ has the same measure as $A$ and it is product-free. Thus, we can find a wide type $p(x)\in S_G(M)$ containing $A(x)$, which is necessarily product-free. Consequently, we have by Theorem \ref{T:Udi} and Lemma \ref{L:StabPF} that $\stab(p)$ is a proper subgroup of bounded index. Since $\stab(p)=(pp^{-1})^2$, it is an $M$-type-definable subgroup of $G^*$ and it is contained in $(A^*(A^*)^{-1})^2$. Thus, a standard compactness argument yields the existence of infinitely many formulas $X_i(x)$ with parameters over $M$ such that
$$
\stab(p)=\bigcap_i X_i
$$
with the following properties: each set $X_i$ is symmetric, contains the identity, it is contained in $(A^*(A^*)^{-1})^2$, we have that $X_{i+1}^2 \subseteq X_i$ and  the group $G^*$ is covered by $k_i$ translates of $X_i$. Furthermore, since $\stab(p)$ is a proper subgroup, we can take $X_1$ properly contained as well. 

In addition, we also have by Theorem \ref{T:Udi} that $pp^{-1}p$ is a coset of $\stab(p)$, say $\stab(p)a$ with $a\in p$, and that $\stab(p)\setminus pp^{-1}$ is not wide. The latter implies that $\stab(p)\setminus A^*(A^*)^{-1}$ is not wide and so there must be some $X_j$ witnessing that $\mu(X_j\setminus A^*(A^*)^{-1})=0$.  Without loss, we may assume that $X_j=X_1$ and note that in particular, we have that $\mu(X_1\setminus A^*(A^*)^{-1})<\epsilon$. On the other hand, since the type-definable set $\stab(p)a$ is included in $A^*(A^*)^{-1}A^*$, again by compactness we find some $X_{k}$ such that $X_ka \subseteq A^*(A^*)^{-1}A^*$. After renaming the sets $X_i$, we can assume that $X_k=X_1$. Finally, observe that $X_ia\cap A^*$ has positive measure, say $\mu(X_ia\cap A^*)=\eta_i$, since both $X_ia$ and $A$ contain $p$. 

Let $X_i$ also denote the interpretation of $X_i(x)$ in $M$. 
Altogether, since the properties we are interested in are all expressible in first-order, we have found formulas $X_i(x)$ with the property that
$$
X_{m} \subseteq X_{m-1} \subseteq \ldots \subseteq X_1 \subseteq (AA^{-1})^2
$$
and such that $X_i=X_i^{-1}$, $1\in X_i$, $X_{i+1}^2\subseteq X_i$, each $X_i$ covers $G$ with $k_i$ many translates and moreover $\mu(X_1\setminus AA^{-1})<\epsilon$. Furthermore, we can find some $a'\in A$ such that $\mu(X_ma'\cap A)\ge \eta_m$ and that $X_ma'\subseteq AA^{-1}A$. Therefore, using \L os' Theorem (for internal sets) we see that the internal sets $X_i$ induce subsets $X_{n,i}$ of $G_n$ satisfying the same properties for $\U$-almost all $n$, contradicting our choice of the $G_n$ and $A_n$. This finishes the proof.
\end{proof}

\subsection{On the question of Babai and S\'os} 
At this point, we can easily answer the original question of Babai and S\'os, concerning the family of finite simple non-abelian groups of a given Lie type and rank, refuted by Gowers in \cite[Corollary 3.4]{Gow}. 

\begin{proposition}\label{P:Simple1}
Let $c>0$ and let $G$ be a finite simple non-abelian group of Lie type of Lie rank $r$. There exists an integer $n=n(r,c)$ such that if $G$ contains a product-free set of size at least $c|G|$, then $|G|\le n$.
\end{proposition}
\begin{proof}
For otherwise, assume that there is a constant $c>0$ and an integer $r$ such that for each positive integer $n$ we can find a finite simple non-abelian group $G_n$ of Lie type of Lie rank $r$ containing a product-free set $A_n$ with $|A_n|\ge c|G_n|$. Fix a non-principal ultrafilter $\U$ on $\N$ and consider the ultraproduct $G = \prod_\U G_n$  of $(G_n)_{n\in\N}$ with respect to $\U$, seen as a nonstandard finite group living in a structure $M$. Let $A=\prod_\U A_n$, an internal product-free subset of $G$. Note that $A$ is indeed definable in $M$ and moreover it satisfies $|A|\ge c|G|$.  

Let $M^*$ be some very saturated elementary extension of $M$ and let $G^*$ and $A^*$ be the interpretaion in $M^*$ of the formulas $G(x)$ and $A(x)$, respectively. Notice that $A^*$ has the measure at least $c$ and it is product-free. Thus, we can find a complete wide type $p(x)\in S_{G}(M)$ containing the formula $A(x)$, which is necessarily product-free. Using Theorem \ref{T:Udi} and Lemma \ref{L:StabPF}, we see that $\stab(p)$ is a proper normal subgroup of $G^*$ of bounded index. 

On the other hand, we know by a result of Point \cite[Corollary 1]{Poi} that $G$  is a simple group. This means that the conjugacy class $a^G$ generates the whole group $G$ for every non-trivial element $a\in G$, that is to say
$$
G = \bigcup_{n\in \N} \left\{x_1\cdot \ldots \cdot x_n : x_1,\ldots,x_n\in a^G \cup (a^{-1})^G \cup\{1\} \right\}.
$$
As $G$ is $\aleph_1$-saturated, there exists some integer $k$  with the property that 
$$
G = \bigcup_{n\le k} \left\{x_1\cdot \ldots \cdot x_n : x_1,\ldots,x_n\in a^G \cup (a^{-1})^G \cup\{1\} \right\}
$$
for every non-trivial element $a$ of $G$. Since this is an elementary property in the language of groups, the same is true of $G^*$ and so $G^*$ is a simple group. Hence, we get that $\stab(p)$ is the trivial group and so $G^*$ has bounded size, a contradiction as $\kappa$ was taken to be arbitrarily large.
\end{proof}

Let us remark that in \cite{Gow}, Gowers answers negatively the question of Babai and S\'os for the family of all finite simple non-abelian groups. This can be deduced from Theorem 4.6 and 4.7 there. To obtain this more general result, note that the same proof as above will work for any ultraproduct without bounded index subgroups. For an ultraproduct of finite simple non-abelian groups, one can get this by following the lines of the proof of Theorem 3.1 of \cite{Pil1} due to Pillay. However, his argument relies on the work of Liebeck and Shalev \cite[Theorem 1.1]{LiSh} which depends on the Classification of Finite Simple Groups. Instead, we use the structure of approximate subgroups to do this.

\subsection{Abelian-by-bounded quotients}
In addition to the stabilizer theorem presented in the previous section, we need the structure theorem of approximate subgroups \cite{BGT}. Consequently, before proceeding we briefly recall the definition and some results concerning approximate subgroups.

\begin{definition}
Let $k\ge 1$. A subset $A$ of a group $G$ is a $k$-approximate subgroup if it is symmetric,  contains the identity and $A^2$ is contained in $XA$ for some symmetric subset $X$ of size at most $k$. 
\end{definition}

Of course, a subgroup is an example of an approximate subgroup but there are many other examples. For instance, any large set with respect to the counting measure is an approximate subgroup. This can be deduce from the following version of Ruzsa's Covering Lemma, see for instance \cite[Fact 5]{MaWa}. Since its proof is short, we include if for completeness. 

\begin{fact}\label{F:AS}
Let $G$ be an ultraproduct of finite groups and let $k\ge 1$. If $A$ is an internal subset of $G$ satisfying $k|A|\ge |G|$, then the set $AA^{-1}$ is a $k$-approximate subgroup. In fact, the set $AA^{-1}$ is left $k$-generic, that is to say the whole group $G$ is covered by $k$ many translates of $AA^{-1}$.
\end{fact}
\begin{proof}
The set $AA^{-1}$ is symmetric and contains the identity element. Now, take $X$ to be a maximal subset of $G$ such that $xA\cap yA =\emptyset$ for every two distinct elements $x,y\in X$. We see that
$$
|X|\mu(A)= \mu(XA) \le 1 \le k\mu(A),
$$
implying that $|X|\le k$. In addition, we get that $G$ and hence $(AA^{-1})^2$ are contained in $XAA^{-1}$ by maximality of $X$. 
\end{proof}

We need \cite[Theorem 4.2]{BGT} due to Breuillard, Green and Tao, which corresponds to the nonstandard version of the main result of their paper, see also \cite[Theorem 1.6]{BGT}. In fact, we state it in a weaker form which is enough for our purposes since we only aim towards qualitative statements concerning finite group. We also refer the reader to \cite{vdD} for an excellent survey as well as to some unplished notes of Hrushovski \cite{Hrus2} for an entirely model-theoretic treatment of the following result.

\begin{theorem}[Breuillard, Green and Tao]\label{T:BGT}
Let $G$ be an ultraproduct of finite groups and let $A$ be an internal $k$-approximate subgroup. Then there exists an internal subgroup $G_0$ of $\langle A \rangle$ of finite index in $G$ and an internal normal subgroup $N$ of $G_0$ contained in $A^4$ such that the group $G_0/N$ is nilpotent.
\end{theorem}
%In the statement, a group $H$ is said to be $d$-nilpotent if it is generated by elements $h_1,\ldots,h_d$ such that $[h_i,h_j] \in \langle h_1,\ldots,h_{i-1}\rangle$ for $1\le i<j\le d$. In particular, a $d$-nilpotent group  is nilpotent of class at most $d$, since each element $h_i$ belongs to the $i$th iterated center $Z_i(H)$ of $H$. 

Notice that this structure theorem can be trivialised if $A^4=G$, since then one can take $G_0=N=G$. Our aim is to apply Theorem \ref{T:BGT} to a set of the form $AA^{-1}$ where $A$ is a given product-free subset. However, {\it a priori} the set $(AA^{-1})^2$ might be equal to the whole group; for instance this is easily seen to happen in cyclic finite groups when $A$ is a large enough product-free set. Nevertheless, as we see, one can circumvent this issue by using an easy compactness argument when working in the ultraproduct.

\begin{proposition}\label{T:GenCase}
Let $G$ be an ultraproduct of finite groups containing a product-free set $A_0$ of positive measure. Then there exists an internal symmetric subset $A$ of $A_0A_0^{-1}$ which is left generic, an internal subgroup $G_0$ of $\langle A \rangle$ of finite index in $G$ and an internal normal subgroup $N$ of $G_0$ with the following properties: 
\begin{enumerate}[$(i)$]
\item the set $A^4$ contains $N$ and it is properly contained in $G$, and
\item the group $G_0/N$ is nilpotent.
\end{enumerate}
\end{proposition}
\begin{proof}
As usual, we see $G$ as a nonstandard finite group defined in a suitable structure $\bar M$. Let $M$ be an  elementary substructure of $\bar M$ of countable size and let $p(x)\in S_G(M)$ be a wide type containing the formula $A_0(x)$. Since $A_0$ is product-free, so is $p$ and consequently $pp^{-1}$ is contained in a proper subgroup of $G$ by Lemma \ref{L:StabPF}. In particular, we then have that $(pp^{-1})^4\subsetneq G$ and consequently, a compactness argument yields the existence of an $M$-definable superset $B$ of $p$, contained in $A_0$, such that $(BB^{ -1})^4 \subsetneq G$. 
Note that $B$ is an internal set and has positive measure, since it contains $p$. Set $A=BB^{-1}$. By Fact \ref{F:AS}, we see that $A$ is a $k$-approximate subgroup, for a suitable positive integer $k$, and also that it is left $k$-generic. %In particular, the subgroup $\langle A \rangle$ of $G$ has index at most $k$. Consequently, we then have that $\langle A \rangle$ and its complement are type-definable sets and hence definable, by compactness. 

 Using now Theorem \ref{T:BGT}, we obtain  an internal subgroup $G_0$ of $\langle A \rangle$ of finite index in $G$ and an internal normal subgroup $N$ of $G_0$ satisfying the desired properties. 
\end{proof}

As an application we get the following statement for finite groups. 

\begin{cor}
For any $c>0$, there is some $m=m(c)$ such that the following holds. Suppose that $G$ is a finite group containing a product-free subset of size at least $c|G|$. Then, there exists a non-perfect subgroup $H$ of $G$ of index at most $m$.
\end{cor}
\begin{proof}
Otherwise, there is a constant $c>0$ such that for every $n$ we can find a finite group $G_n$ containing a product-free subset $A_n$ with $|A_n|\ge c|G_n|$ and which in addition does not contain a non-perfect subgroup $H_n$ of index at most $n$. Fix now a non-principal ultrafilter $\U$ on $\N$ and consider the ultraproduct $G = \prod_\U G_n$  of $(G_n)_{n\in\N}$ with respect to $\U$ and let $A=\prod_\U A_n$, an internal product-free subset of $G$. The previous result yields the existence of an internal subgroup $H$ of $G$ of finite index, say $m$, and an internal normal subgroup $N$ of $H$ such that $H/N$ is non-trivial and nilpotent. Using \L os Theorem, we see that $G_n$ has a non-perfect subgroup of index at most $m$ for $\U$-almost all $n$, giving the desired contradiction.
\end{proof}
Note that in the statement above one can demand the subgroup $H$ to be normal. This is standard. Namely, given a subgroup $H$ of a group $G$ consider the group homomorphism $G\rightarrow \mathrm{Sym}(G/H)$ defined by $g\mapsto \tau_g$, where $\tau_g$ is the permutation of the coset space $G/H$ mapping $xH$ to $gxH$. It then follows that kernel of this homomorphism is a normal subgroup of $G$ contained in $H$ and whose index is bounded above by $|\mathrm{Sym}(G/H)|$.

As a consequence, combining this observation with the previous corollary we answer negatively the question of Babai and S\'os for the family of finite simple non-abelian groups, extending Proposition \ref{P:Simple1}. As previously remarked, this is a mere corollary of Theorem 4.6 and 4.7 of \cite{Gow}.

\begin{cor}\label{C:Simple2}
For any constant $c>0$, there is only a finite number of finite simple non-abelian groups $G$ containing a product-free set of size at least $c|G|$.
\end{cor}

%In this section we go back to the standard finite setting. More precisely, we use the results of the previous section to derive some structure theorem on arbitrary finite groups admitting large product-free sets. Before dealing with the general case, we shall analyse finite groups of finite exponent.

\section{Compactifications of ultraproducts}\label{S:DefCom}

In this section we relate the existence of product-free sets of positive measure with certain definable compactifications. We first recall the definition of compactification of a (discrete) group.

\begin{definition}
A {\em compactification} of a (discrete) group $G$ is a compact Hausdorff topological group $\H$ and a group homomorphism $\rho : G \rightarrow \H$ with dense image.
\end{definition} 

Among the possible compactifications of $G$, there is a universal one called the {\em Bohr compactification of $G$} which is usually denote by $bG$. Namely, a compactification $b : G \rightarrow bG$ such that for every compactification $\rho : G \rightarrow \H$ there is a unique continuous surjection $\tau : bG\rightarrow \H$ such that $\rho = \tau \circ b$. 

Next, we shall explain how one can describe and generalize Bohr compactifications from a model theoretic point of view.

Let $G$ be a definable group in some structure $M$, not necessarily saturated.

\begin{definition}
A {\em definable group compactification} of $G$ with respect to $M$ is a compact Hausdorff topological group $\H$ and a group homomorphism $\rho: G\rightarrow \H$ with a dense image satisfying the following continuity property: given two disjoint closed subsets $C_1$ and $C_2$ of $\H$ then there exists a definable subset $C$ of $G$ such that 
$$
\rho^{-1}(C_1)\subseteq C \ \text{ and } \ \rho^{-1}(C_2)\subseteq G\setminus C
.$$
\end{definition}

When all subsets of $G$ are definable in $M$, then a definable compactification of $G$ is nothing else than a compactification. But in general, not all subsets of $G$ might be definable. 
In \cite[Proposition 3.4]{GPP}, it is shown that there is always a universal definable compactification, yielding an alternative proof of the existence of the Bohr compactification. To describe it, consider a $\kappa$-saturated elementary superstructure $M^*$ of $M$, for a sufficiently large cardinal $\kappa$, and let $G^*$ be the interpretation of the formula defining $G$ in $M^*$. Consider the group $(G^*)_{M}^{00}$, the smallest subgroup of $G^*$ which is type-definable over $M$ and has bounded index in $G^*$. Then, the universal definable compactification is given by the group $G^*/(G^*)_M^{00}$ equipped with the logic topology and the natural homomorphism from $G$ to $G^*/(G^*)_M^{00}$ induced by the identity embedding from $G$ into $G^*$.  

\subsection{Bohr compactification}
Concerning ultraproducts of finite groups, Pillay in \cite[Theorem 3.1]{Pil1} proved that an ultraproduct of finite simple non-abelian groups has trivial Bohr compactification. More recently, Nikolov, Schneider and Thom \cite[Theorem 8]{NST} have extended Pillay's result, answering a question of Zilber.

\begin{theorem}[Nikolov, Schneider and Thom]
If $G$ is an ultraproduct of finite groups, then the identity component $(bG)^0$ of $bG$ is commutative.
\end{theorem}

As an easy consequence, we can characterise those ultraproducts of finite groups with a trivial Bohr compactification. Let us remark that a group has trivial Bohr compactification if and only if it is type-absolutely connected in the sense of Gismatullin \cite[Definition 3.2]{Gis}.

\begin{cor}\label{C:NST}
Let $G$ be an ultraproduct of finite groups. Then $bG$ is trivial if and only if $G$ is perfect and $G$ has no finite index subgroup.
\end{cor}
\begin{proof}
Suppose that $bG$ is trivial. Thus, the group $G$ has no finite index subgroup. To see that $G$ is perfect we refer the reader to \cite[Proposition 3.5]{Gis} for an easy proof. 

For the converse, assume that $G$ is perfect and that $G$ has no finite index subgroup. Note that the latter means that the identity component $(bG)^0$ of $bG$ equals $bG$, and the former that
$$
G = \bigcup_{n\in \N} \big\{[a_1,b_1]\cdot \ldots \cdot [a_n,b_n] :a_1,b_1,\ldots,a_n,b_n\in G \big\}.
$$
Note that each subset involved in this union is definable in the pure language of groups. Since $G$ is an $\aleph_1$-saturated group, a model-theoretic compactness argument yields the existence of some positive integer $m$ such that any element of $G$ can be written as the product of at most $m$ many commutator elements. %Note that this is a property expressible in the pure language of groups.

Now, consider the structure $M$ whose domain is $G$ together with its group structure, as well as a unary relation for each subset of $G$. Let $M^*$ be an elementary extension of $M$, sufficiently saturated, and let $G^*$ denote the interpretation of the formula $G(x)$ in $M^*$. As pointed out before, the Bohr compactification of $G$ is precisely the group $ G^*/ (G^*)_M^{00}$ equipped with the logic topology. Thus 
$$
(bG)^0 =  bG = G^*/ (G^*)_M^{00}
$$
and therefore the derived subgroup of $G^*$ is contained in $(G^*)_M^{00}$, by the aforementioned result of Nikolov, Schneider and Thom. On the other hand, since $G$ and $G^*$ are elementary equivalent (as pure groups), every element of $G^*$ can also be written as the product of at most $m$ many commutator elements of $G^*$. Hence, the group $G^*$ is perfect yielding that $bG = 1$, as desired.
\end{proof}

As an easy observation, note that this result (or its proof) yields that an ultraproduct of finite groups $(G_n)_{n\in \N}$ has non-trivial Bohr compactification whenever the groups $G_n$ are not perfect or the commutator width is not uniformly bounded among the $G_n$, {\it i.e.} there is no $k$ such that every commutator element of each $G_n$ can be written as the product of $k$ commutator elements.

\subsection{Internal compactification} Another natural compactification of an ultraproduct of finite groups is the case when the definable sets and the internal ones coincide. This motivates the following definition:

\begin{definition}
Let $G$ be an ultraproduct of finite groups. By an internal compactification of $G$ we mean a definable compactification of $G$ with respect to a structure $M$ such that the Boolean algebras of the internal and definable sets coincide.
\end{definition}

We denote the universal internal compactification of $G$ by $\mathrm{int}\, bG$.
For those readers not well-versed in model theory, one can explicitly construct the universal internal compactification by taking the completion of $G$ with respect to the
topology on $G$, whose neighbourhoods of the identity are the internal subsets $U$ of $G$ which admit a sequence $(U_n)_{n\in \N}$ of internal subsets of $G$ with $U_0=U$ satisfying for each $n$ that $U_{n+1}^2 \subseteq U_n$ and $U_{n+1}$ is symmetric and left generic. 

In \cite[Theorem 2.2]{Pil1}, Pillay proved using the structure theorem of approximate subgroups that the identity component $(\mathrm{int}\, bG)^0$ of $\mathrm{int}\, bG$ is commutative. In fact, notice that Proposition \ref{T:GenCase} (which also relies on the work of Breuillard, Green and Tao) is a variant of this. Next, we shall prove that these conditions are indeed equivalent, and relate them to the notion of quasirandom, originated in the work of Gowers, which we recall now. 

\begin{definition}
A group $G$ is {\em $d$-quasirandom} for some parameter $d\ge 1$ if all non-trivial unitary representations of $G$ have dimension at least $d$.
\end{definition}

Using the Peter-Weyl Theorem, we remark that a group $G$ has a trivial Bohr compactification if and only if it is $d$-quasirandom for every $d\ge 1$. Usually, an infinite group which is $d$-quasirandom for every $d\ge 1$ is also called {\em minimally almost periodic}. 

A more natural notion of quasirandomness for ultraproducts of finite groups, encompassing the structure coming from the finite setting, is the following. This corresponds to Definition 31 of \cite{BeTa}.

\begin{definition}
An {\em ultra quasirandom} is an ultraproduct $G=\prod_\U G_n$ of finite groups with the property that for every $d\ge 1$, the groups $G_n$ are $d$-quasirandom for $\U$-almost all $n$.
\end{definition}

In general an ultra quasirandom group might not be $d$-quasirandom for some $d\ge 1$. In fact, in \cite{Yan} Yang provides an example of an ultra quasirandom group which is not even $2$-quasirandom.  This example appears there as Example 1.7 and it is attributed to Pyber. We present it at the end of this section to describe an ultraproduct $G$ of finite groups with the property that $\mathrm{int}\, bG = 1$ but $bG\neq 1$.

We prove our main result. Note that the statement without condition $(b)$ follows from the finite setting. So strictly speaking, the only new equivalence is given by condition $(b)$.

\begin{theorem}\label{T:Main}
Let $G$ be an ultraproduct of finite groups. Then the following are equivalent:
\begin{enumerate}[$(a)$]
\item Every internal product-free subset of $G$ has measure zero.
%\item There is no definable group compactification $\rho:G\rightarrow \H$ of $G$ to a separable compact Hausdorff group $\H\neq 1$.
\item The universal internal compactification $\mathrm{int}\, bG$ of $G$ is the trivial group.
\item If $A$ and $B$ are internal non-null subsets of $G$, then $\mu(AB)=1$.
\item If $A,B$ and $C$ are internal non-null subset of $G$, then $G=ABC$.
\item There is no internal proper finite quotient nor an internal abelian one.
\item The group $G$ is ultra quasirandom.
\end{enumerate} 
\end{theorem}
\begin{proof}
Let $G$ denote an ultraproduct of finite groups. %As explained in Section 2 it can naturally be seen as a nonstandard finite group, living canonically as a definable object in suitable structure $\bar M$.

$(a)\Rightarrow (b)$. Assume that $G$ has no internal product-free set of positive measure. Suppose to get a contradiction that $\mathrm{int}\, bG$ is non-trivial. Let $M$ be the $1$-sorted structure with sort $G$ and with primitive relations the graph of the group multiplication as well a unary relation for each internal subset of $G$. Let $M^*$ be a sufficiently saturated elementary extension of $M$ and let $G^*$ denote the interpretation in $M^*$ of the formula $G(x)$ defining $G$. As explained above, we know that the $\mathrm{int}\, bG$ is given by the quotient of $G^*$ modulo $(G^*)_{M}^{00}$. Any proper coset of $(G^*)_{M}^{00}$ is a product-free set of $G^*$ and so we can find an $M$-definable subset $A^*$ of $G^*$, containing a proper coset of $(G^*)_{M}^{00}$, which is product-free. Note then that $A^*$ is left generic, by model theoretic compactness. 

Now, let $A$ be the interpretation in $M$ of the formula defining the $M$-definable set $A^*$. This is clearly an internal product-free set. Furthermore, since saying that the union of a concrete number of translates cover the whole group is an elementary statement, the set $A$ is also left generic. However, this yields that $A$ has positive measure,  a contradiction.

$(b)\Rightarrow (c)$. Suppose that $G$ has no non-trivial internal compactification, and let $A$ and $B$ be two internal subsets of $G$ of positive measure. By Lemma \ref{L:Intersection}, there exists some element $g\in G$ with the property that the set $A\cap gB^{-1}$ has positive measure and denote this internal set by $C$.

Now, regard $G$ as a nonstandard finite group defined in some structure $M$. Let $M^*$ be a sufficiently saturated elementary extension of $M$ and let $G^*$ and $C^*$ be the interpretation in $M^*$ of the formula $G(x)$ and $C(x)$ respectively. Fix some wide type $p(x)\in S_G(M)$ extending $C(x)$. Since the assumption yields that $G^*$ equals to $(G^*)_M^{00}$, using Theorem \ref{T:Udi} we see that $G^*\setminus pp^{-1}$ is contained in a union of $M$-definable sets of measure zero. Thus, any internal superset of $pp^{-1}$, for instance $C^*C^{*-1}$, must have full measure. Note that the value of the measure only depends on the formula and consequently, we get:
$$
\mu(AB)=\mu(ABg^{-1})\ge \mu(CC^{-1}) =1.
$$

$(c)\Rightarrow (d)$. Assume now that $(c)$ holds and let $A,B$ and $C$ be three wide subsets of $G$. To see that $G=ABC$, take an arbitrary element $g$ of $G$. Since $\mu(AB)=1$, an easy computation yields that 
$$
\mu(AB\cap gC^{-1}) = \mu(gC^{-1}) = \mu(C)
$$ 
and so $AB\cap gC^{-1}$ is non-empty, yielding $g\in ABC$.

$(d)\Rightarrow (a)$. Given an internal set $A$ of positive measure, set $B=A$ and $C=A^{-1}$. By assumption, we have that the identity element of $G$ belongs to $AAA^{-1}$ and so $A$ is not product-free. 

$(a)\Leftrightarrow (e)$. If $G$ has an internal subgroup of finite index, then any proper coset is an internal product-free set of positive measure. Additionally, if $G$ has an internal abelian quotient $\bar G$, then it has an internal product-free set $\bar A$ of positive measure. In fact, one can take $\bar A$ in such a way that $|\bar A|\ge 2 |\bar G|/7$, by \cite[Corollary 7.8]{BaSo}. We then obtain an internal product-free subset $A$ of $G$ with $|A|\ge 2|G|/7 $; namely, take $A$ to be the pre-image of $\bar A$ under the natural projection from $G$ onto $\bar G$. Altogether, we see that $(a)$ implies $(e)$. The converse follows from Proposition \ref{T:GenCase} (or alternatively from the work of Gowers).

Finally, we see the equivalence between condition $(f)$ and the rest. In fact, the equivalence $(e)\Leftrightarrow(f)$ is precisely \cite[Theorem 4.8]{Gow}. One direction is easy. For the other, one can alternatively use a classical theorem of Jordan asserting that a finite group of $\mathrm{U}_d(\mathbb C)$ has an abelian subgroup whose index only depends on $d$. %, see \cite[Section 2]{BG} for a proof. 
\end{proof}

It is routine to check, taking ultraproducts and using \L os' Theorem (for internal sets), that the above result can be written as follows for a finite group. We omit the proof. 

\begin{cor}\label{C:MainFinite}
Let $G$ be a finite group of order $n$. Then the following statements are
equivalent, in the sense that for any constant $c$ there exist other constants only depending on $c$ such that if $G$ satisfies one statement with respect to the constant $c$, then $G$ satisfies all the others statements with respect to these other constants.
\begin{enumerate}[$(a)$]
\item There is a product-free subset of $G$ of size at least $c_1n$.
%\item For any two subsets $A$ and $B$ of $G$ of size at least $c_2n$ we have that $|G\setminus AB|<O_{c_2}(1)|G|$.
\item There are three subsets $A, B$ and $C$ of $G$ of size at least $c_2n$ such that $G \neq ABC$.
\item  There is a non-perfect subgroup of $G$ of index at most $1/c_3$.
\item The group $G$ is not $1/c_4$-quasirandom.

%\item For any $m\ge 2$ and every subset $A$ of $G$ of size at least $c_{5,m}$, the set  
\end{enumerate}
\end{cor}

In fact, this result corresponds to a qualitative version of results due to Gowers in  \cite{Gow}, where the constants are determined. Furthermore, answering a question of Gowers, Nikolov and Pyber \cite[Theorem 3]{NikPyb} shown that the statements above are polynomially-equivalent using the Classification of Finite Simple Groups. %Therefore, strictly speaking, the new result is condition $(b)$. 
Moreover, the equivalence with condition $(c)$ in the previous theorem has the following easy consequence.

\begin{cor}\label{C:AsymProd}
For any $c,\epsilon>0$ there exists some $d=d(c,\epsilon)$ such that the following holds. For any two subsets $A$ and $B$ of a $d$-quasirandom finite group $G$ of order $n$, we have that $|AB| \ge (1-\epsilon)|G|$ whenever $A$ and $B$ have size at least $cn$.
\end{cor}
\begin{proof}
Otherwise, there are two constans $c>0$ and $\epsilon>0$ such that for every $n$ we can find some finite $d_n$-quasirandom subgroup $G_n$ in a way that $\lim_n d_n=\infty$ and with the property that $G_n$ contains two subsets $A_n$ and $B_n$ of size at least $c|G_n|$ satisfying $|A_nB_n|<(1-\epsilon)|G_n|$. Consider the ultra quasirandom group $G=\prod_\U G_n$ and set $A=\prod_\U A_n$ and $B=\prod_\U B_n$. We see that $A$ and $B$ are non-null internal sets and so $\mu(AB)=1$ by Theorem \ref{T:Main}. However, this contradicts the fact that $|AB|<(1- \epsilon)|G|$, by construction. 
\end{proof}

\subsection{An example} To finish this section, we see that Pyber's example (see \cite[Example 1.7]{Yan}) yields the existence of an ultraproduct of finite groups where the Bohr compactification and the universal internal compactification do not agree.

\begin{example}
Let $p\ge 5$ be a prime and let $(G_n)_{n\in\N}$ be an infinite sequence of finite groups $G_n$ satisfying the property that $G_n$ is a perfect group with an element $a_n$ which cannot be written as the product of $n+1$ commutator elements, and that the only simple quotient of $G_n$ is $\mathrm{PSL}_2(\mathbb F_{p^n})$. Such a family of finite groups exists by for instance \cite[Lemma 2.1.10]{HoPl}. 

Now, let $G$ be the ultraproduct $\prod_\U G_n$ of $(G_n)_{n\in\N}$ with respect to some non-principal ultrafilter $\U$ on $\N$. Let $a$ be the ultraproduct of the sequence $(a_n)_{n\in\N}$ and note that it cannot be written as the product of $m$ commutator elements for any $m$. Consequently, the group $G$ is not perfect and so its Bohr compactification $bG$ is non-trivial by Corollary \ref{C:NST}, say. On the other hand, we see that any proper normal subgroup of $G_n$ has index at least $|\mathrm{PSL}_2(\mathbb F_{p^n})|$. Thus, using \L os' Theorem (for internal sets) we see that the group $G$ has no abelian-by-finite internal proper quotient and so its universal internal compactification $\mathrm{int}\, bG$ is the trivial group, by Theorem \ref{T:Main}.
\end{example}

\section{Profinite compactifications} \label{S:FinExp}

In the previous section we have shown, among other things, the equivalence between having an internal product-free subset of positive measure and admitting a non-trivial internal compactification. Thus, it is natural to expect that some topological or algebraic properties of $\mathrm{int}\, bG$ have some impact on the structure of the group, at the level of internal sets. 
In particular, given an ultraproduct $G$ of finite groups such that $\mathrm{int}\, bG$ is profinite, we have that $G$ has an internal subgroup of finite index if and only if $\mathrm{int}\, bG$ is non-trivial. Namely, if $\mathrm{int}\, bG$ is non-trivial and profinite,  it admits a fundamental system $\mathcal F$ of open neighbourhoods $U$ of the identity such that each $U$ is an open normal subgroup  and $\bigcap_{U\in \mathcal F} U = 1$. Since an open subgroup  is closed of  finite index, the preimage of each subgroup $U$ is an internal subgroup of finite index, by definition of internal group compactification.

\begin{cor}
Let $G$ be an ultraproduct of finite groups such that $\mathrm{int}\, bG$ is profinite. Then the following condition is equivalent to  conditions $(a)-(f)$ of Theorem \ref{T:Main}:
\begin{enumerate}[$(g)$]
\item The group $G$ has no internal proper finite quotient.
\end{enumerate}
\end{cor}
 
Since a proper coset of a subgroup is a product-free set, we see that the original question of Babai and S\'os has a positive answer in a strong way when one restricts the attention to families of finite groups whose ultraproduct admits a non-trivial profinite internal compactification. In fact, 
using Hrushovski's stabilizer theorem we obtain the following structure theorem, asserting that every product-free set is closely related to  a subgroup.

\begin{theorem}\label{T:FinExp}
Let $G$ be an ultraproduct of finite groups and assume that $\mathrm{int}\, bG$ is profinite. Suppose that $G$ contains an internal product-free set $A$ of positive measure. Then, there is an internal normal proper subgroup $H$ of $G$ of finite index which satisfies the following properties:
\begin{enumerate}[$(a)$]
\item it is contained in $(AA^{-1})^2$ with $\mu(H\setminus AA^{-1})=0$, and
\item a coset $C$ of $H$ is contained in $AA^{-1}A$ and the internal set $C\cap A$ has positive measure.
\end{enumerate}
\end{theorem}
\begin{proof}
Regard $G$ as a nonstandard finite group in some structure $M$ and let $M^*$ be a sufficiently saturated elementary extension of $M$. Let $G^*$ and $A^*$ be the interpretation in $M^*$ of the formula $G(x)$ and $A(x)$ respectively. Fix some wide type $p(x)\in S_G(M)$ extending $A(x)$ and note that $p$ is product-free as so is $A$. Using Theorem \ref{T:Udi} and Lemma \ref{L:StabPF}, we see that $\stab(p)$ is a proper $M$-type-definable subgroup of $G^*$ of bounded index. Furthermore, we have $\stab(p)=(G^*)_M^{00}$. It then follows that $G^*/\stab(p)$ is a non-trivial profinite group, equipped with the logic topology and consequently, as explained above, we get that $\stab(p)$ is the intersection of $M$-definable subgroups of finite index. 

On the other hand, we know by Theorem \ref{T:Udi} and Lemma \ref{L:StabPF} that the following holds:
\begin{enumerate}[$(i)$]
\item The set $\stab(p)\setminus pp^{-1}$ is contained in $(pp^{-1})^2$ and in a union of $M$-definable sets of measure zero. 
\item The set $pp^{-1}p$ is a proper coset of $\stab(p)$, say $\stab(p)u$.
\end{enumerate}
By $(i)$, we see that the $M$-type-definable group $\stab(p)$ is contained in $(A^*(A^*)^{-1})^2$ and in $A^*(A^*)^{-1} \cup C$ for some $M$-definable set $C$ of measure zero. Furthermore, note that $\stab(p)u$ is a subset of $A^*(A^*)^{-1}A^*$ by $(ii)$. Thus, a model-theoretic compactness argument yields the existence of an $M$-definable supergroup $H^*$ of $\stab(p)$ such that
$$
H^* \subseteq (A^*(A^*)^{-1})^2 \cap \left(A^*(A^*)^{-1} \cup C\right) \ \text{ and } \ H^*u\subseteq A^*(A^*)^{-1}A^*.
$$ 
In particular, the subgroup $H^*$ has finite index in $G^*$ and $H^*\setminus A^*(A^*)^{-1}$ has measure zero. In addition, note that the set $H^*u\cap A^*$ has positive measure, since it contains $p$. 

Finally, set $H$ to be the interpretation in $M$ of the formula defining $H^*$. As $M^*$ is an elementary extension of $M$ and the value of the measure only depends on the formulas, we see that $H$ satisfies the desired properties.
\end{proof}
 
Now, consider groups of finite exponent, {\it i.e.} groups satisfying the law $x^m=1$ for some fixed integer $m\ge 1$. We point out that the universal internal compactification of an  ultraproduct of finite groups of a given exponent  is non-trivial and profinite. 

\begin{lemma}
Let $G$ be an ultraproduct of finite groups and assume that $G$ has finite exponent. Then, the group  $\mathrm{int}\, bG$ is a non-trivial profinite group, {\it i.e.}  $\mathrm{int}\, bG \neq 1$ and  $(\mathrm{int}\, bG)^0 = 1$.
\end{lemma}
\begin{proof}
We first prove that $\mathrm{int}\, bG$ is profinite. To see this, note that $\mathrm{int}\, bG $ is a group of finite exponent, by the description given in the previous section. Thus, the claim follows by using Theorem 4.5 of \cite{Ilt}, which asserts that a compact Hausdorff group of finite exponent is profinite.

To show that $\mathrm{int}\, bG$ is non-trivial, assume that $G=\prod_\U G_n$. For each $n$, let $H_n$ be a maximal normal proper subgroup of $G_n$ and note that every quotient $G_n/H_n$ is simple. Set $H$ to denote $\prod_\U H_n$, an internal normal proper subgroup of $G$. 

Suppose that $G$ has exponent $m$. We may distinguish two cases. If $\U$-almost all quotients $G_n/H_n$ are cyclic, then all have at most order $m$. Otherwise, if $\U$-almost all quotients $G_n/H_n$ are simple non-abelian groups, then $\U$-almost all are isomorphic, since by the Classification of Finite Simple Groups there is only a finite number of finite simple non-abelian groups of exponent $m$. Therefore, we get that $G/H$ is a finite group and so $\mathrm{int}\, bG \neq 1$, as desired.
\end{proof}

Therefore, the previous theorem applies to groups of finite exponent. As a consequence, we get the following result in the finite setting.

\begin{cor}\label{C:FinExp}
For any $c,\epsilon >0$ and $m\ge 1$, there exists some integer $k=k(c,\epsilon,m)$ and some constant $\delta=\delta(c,\epsilon,m)$ such that the following holds. Suppose that $G$ is a finite group of order $n$ and exponent $m$ and let $A$ be a product-free set $A$ of size at least $cn$. Then there is a normal proper subgroup $H$ of $G$ of index at most $k$ which satisfies the following properties:
\begin{enumerate}[$(a)$]
\item it is contained in $(AA^{-1})^2$ with $|H\setminus AA^{-1}|<\epsilon |G|$, and
\item some coset $C$ of $H$ is contained in $AA^{-1}A$ with $|C\cap A|\ge \delta|G|$.
\end{enumerate}
\end{cor}
\begin{proof}%[Proof of Theorem \ref{T:Exponent}]
As usual in this kind of proofs, we argue by contradiction. Negating the quantifiers, assume that there exists some constants $c,\epsilon>0$ and $m\ge 1$ such that for each natural number $n$ we can find a finite group $G_n$ of exponent $m$ containing a product-free set $A_n$ with $|A_n|\ge c|G_n|$ but $G_n$ does not contain a normal proper subgroup $H_n$ of index at most $k_n$ such that
\begin{enumerate}[$(i)$]
\item it is contained in $(A_nA_n^{-1})^2$ with $|H_n\setminus (A_nA_n^{-1})|< \epsilon |G_n|$, and
\item a right proper coset $H_nu_n$ of $H_n$ is contained in $A_nA_n^{-1}A_n$ and in addition $|H_nu_n\cap A_n|> (1/n)|G_n|$.
\end{enumerate}    
Now, consider the ultraproduct $G=\prod_\U G_n$ of $(G_n)_{n\in \N}$ and set $A=\prod_\U A_n$, an internal set satisfying $|A|\ge c|G|$. Note that $A$ is product-free and has positive measure, since one can easily see that $\mu(A)\ge c$. Since $\mathrm{int}\, bG$ is profinite by the previous lemma, Theorem \ref{T:FinExp} yields the existence of an internal normal proper subgroup $H$ of $G$ of finite index and a proper coset $Hx$ of $H$ which satisfy
$$
H\subseteq (AA^{-1})^2 \ , \ \mu(H\setminus (AA^{-1}))=0 \ , \ Hx\subseteq AA^{-1}A \ \text{ and } \ \mu(Hx\cap AA^{-1})>0.
$$ 
Set $\gamma=\mu(Hx\cap AA^{-1})$. Let $(x_n)_{n\in \N}$ be a sequence with $x_n\in G_n$ whose ultraproduct is $x$ and let $(H_n)_{n\in \N}$ be a sequence with $H_n$ a subset of $G_n$ such that $H=\prod_\U  H_n$. Hence, using \L os's Theorem (for internal sets), we conclude that for $\U$-almost all $n$ the set $H_n$ is indeed a normal proper subgroup of $G_n$ of index $[G:H]$ satisfying the following properties:
\begin{enumerate}[$(i)$]
\item it is contained in $(A_nA_n^{-1})^2$ with $|H_n\setminus (A_nA_n^{-1})|< \epsilon |G_n|$, and
\item the right proper coset $H_nx_n$ of $H_n$ is contained in $A_nA_n^{-1}A_n$ and satisfies $|H_nx_n\cap A_n|\ge \gamma |G_n|$.
\end{enumerate}   
However, this contradicts the construction of $G_n$ and $A_n$. 
\end{proof}

\bibliographystyle{plain}

\begin{thebibliography}{99}

\bibitem{BaSo}
L. Babai and V. T. S\'os.
\newblock {\em Sidon sets in groups and induced subgraphs of Cayley graphs}, Europ.
J. Combin. 6 (1985), 101--114.

\bibitem{BeTa}
V. Bergelson and T. Tao.
\newblock {\em Multiple recurrence in quasirandom groups}. Geom. Funct. Anal. 24 (2014), no. 1, 1--48. 

\bibitem{Bre} E. Breuillard. \emph{A brief introduction to approximate groups}. Thin Groups and Superstrong Approximation (Eds: E. Breuillard and H. Oh). MSRI Publications 61, 2014.

\bibitem{BGT}
E. Breuillard, B. Green and T. Tao.
\newblock {\em The structure of approximate groups}. Publ. Math. Inst. Hautes \'Etudes Sci. 116 (2012), 115--221. 

\bibitem{ChHr}
G. Cherlin and E. Hrushovski.
\newblock {\em Finite structures with few types}. Annals of Mathematics Studies, 152. Princeton University Press, Princeton, NJ, 2003.

\bibitem{vdD}
L. van den Dries.
\newblock {\em Approximate groups}. Ast\'erisque 367--368 (2015), 79--113.

\bibitem{vdDGold}
L. van den Dries and I. Goldbring.
\newblock {\em Hilbert's 5th problem}. Enseign. Math. 61 (2015), no. 1-2, 3--43.

%\bibitem{GarMacSte}
%D. Garc\'ia, D. Macpherson and C. Steinhorn.
%\newblock {\em Pseudofinite structures and simplicity}. J. Math. Log. 15 (2015), no. 1, 41 pp. 

\bibitem{Gis}
J. Gismatullin.
\newblock {\em Absolute connectedness and classical groups}. Preprint (2012), arXiv:1002.1516v5.

\bibitem{GPP}
J. Gismatullin, D. Penazzi and A. Pillay.
\newblock {\em On compactifications and the topological dynamics of definable groups}. Annals Pure and Appl. Logic. 165 (2014) 552--562.

\bibitem{Gow}
W. T. Gowers.
\newblock {\em Quasirandom groups}. Combin. Probab. Comput. 17 (2008), no. 3, 363--387. 

\bibitem{HoPl}
D. F. Holt and W. Plesken.
\newblock {\em Perfect groups}. With an appendix by W. Hanrath. Oxford Mathematical Monographs. Oxford Science Publications. The Clarendon Press, Oxford University Press, New York, 1989.

\bibitem{Hrus1}
E. Hrushovski.
\newblock {\em Stable group theory and approximate subgroups}. J. Amer. Math. Soc. 25 (2012), no. 1, 189--243. 

\bibitem{Hrus2}
E. Hrushovski.
\newblock {\em Class on Approximate subgroups: Notes on Hilbert $5$ and BGT}. Preprint (2012), http://www.ma.huji.ac.il/~ehud/notesH5-BGT.pdf 

\bibitem{Hrus3}
E. Hrushovski.
\newblock {\em On pseudo-finite dimensions}. Notre Dame J. Form. Log. 54 (2013), no. 3--4, 463--495.

\bibitem{Ilt}
R. Iltis.
\newblock {\em Some algebraic structure in the dual of a compact group}. Canad. J. Math. 20 (1968), 1499--1510.

\bibitem{Ked1}
K. S. Kedlaya.
\newblock {\em Product-Free Subsets of Groups}, Amer. Math. Monthly 105 (1998), no. 10, 900--906.

\bibitem{Ked2}
K. S. Kedlaya.
\newblock {\em Product-free subsets of groups, then and now},  Communicating mathematics, 169--177,  Contemp. Math., 479, Amer. Math. Soc., Providence, RI, 2009. 

%\bibitem{La}
%M. Larsen.
%\newblock {\em Word maps have large image}. Israel J. Math. 139 (2004), 149--156. 

\bibitem{LiSh}
M. Liebeck and A. Shalev.
\newblock {\em Diameters of finite simple groups: sharp bounds and applications}. Ann. of Math. (2) 154 (2001), no. 2, 383--406. 


\bibitem{MaTe}
D. Macpherson and K. Tent.
\newblock {\em Pseudofinite groups with NIP theory and definability in finite simple groups}. Groups and model theory, 255--267, Contemp. Math., 576, Amer. Math. Soc., Providence, RI, 2012. 

\bibitem{MaWa}
J. C. Massicot and F. O. Wagner.
\newblock {\em Approximate subgroups}. J. \'Ec. Polytech. Math. 2 (2015), 55--64. 

\bibitem{MOS}
S. Montenegro, A. Onshuus and P. Simon.
\newblock {\em Groups with f-generics in NTP$2$ and PRC fields}. J. Inst. Math. Jussieu (to appear) 	arXiv:1610.03150. 

\bibitem{NikPyb}
N. Nikolov and L. Pyber. 
\newblock {\em Product decompositions of quasirandom groups and a Jordan type theorem}. J. Eur. Math. Soc. 13 (2011), no. 4, 1063--1077. 

\bibitem{NST}
N. Nikolov, J. Schneider and A. Thom.
\newblock {\em Some remarks on finitarily approximable groups}. J. \'Ec. Polytech. Math. 5 (2018), 239--258.

\bibitem{Pil0}
A. Pillay.
\newblock {\em Type-definability, compact Lie groups, and o-minimality}. J. Math. Log. 4 (2004), no. 2, 147--162. 

\bibitem{Pil1}
A. Pillay.
\newblock {\em Remarks on compactifications of pseudofinite groups}. Fund. Math. 236 (2017), no. 2, 193--200.

\bibitem{PiScWa}
A. Pillay, T. Scanlon and F. Wagner.
\newblock {\em Supersimple fields and division rings}. Math. Res. Lett. 5 (1998), no. 4, 473--483. 

\bibitem{Poi}
F. Point.
\newblock {\em Ultraproducts and Chevalley groups}. Arch. Math. Logic 38 (1999), no. 6, 355--372. 

\bibitem{San}
T. Sanders.
\newblock {\em On a nonabelian Balog-Szemer\'edi-type lemma}. J. Aust. Math. Soc. 89 (2010), no. 1, 127--132. 

\bibitem{Yan}
Y. Yang.
\newblock {\em Ultraproducts of quasirandom groups with small cosocles}. J. Group Theory 19 (2016), no. 6, 1137--1164. 



\end{thebibliography}

\end{document}